\documentclass[11pt,a4paper]{amsart}
\usepackage{amssymb}
\usepackage[a4paper,centering,twoside,textwidth=6in]{geometry}

\newcommand{\al}{\alpha}
\newcommand{\be}{\beta}
\newcommand{\ga}{\gamma}

\newcommand{\la}{\lambda}
\newcommand{\ka}{\kappa}
\newcommand{\eps}{\epsilon}

\DeclareMathOperator{\PSL}{PSL} \DeclareMathOperator{\PGL}{PGL}
\DeclareMathOperator{\SL}{SL} 
\DeclareMathOperator{\tr}{tr}

\newcommand{\Tr}{\mathcal{T}}
\newcommand{\Pri}{\mathcal{P}}
\newcommand{\cC}{\mathcal{C}}

\newcommand{\BF}{\mathbb{F}}

\theoremstyle{plain}
\newtheorem{prop}{Proposition}
\newtheorem{lemma}[prop]{Lemma}
\newtheorem{cor}[prop]{Corollary}

\newtheorem{thm}[prop]{Theorem}

\theoremstyle{definition}
\newtheorem{defn}{Definition}

\newtheorem{example}[prop]{Example}
\newtheorem{rem}[prop]{Remark}

\title{Expansion of conjugacy classes in $\PSL_2(q)$}
\author{Shelly Garion}
\thanks{The author was supported by the SFB 878 ``Groups, Geometry and Actions''.}

\address{Fachbereich Mathematik und Informatik, Universit\"at
M\"unster, Einsteinstrasse 62, D-48149 M\"unster, Germany}
\email{shelly.garion@uni-muenster.de}
\subjclass[2000]{20D06}

\begin{document}

\maketitle

\begin{abstract}
For any conjugacy class $\cC$ in $G=\PSL_2(q)$ we compute $\cC^2$
and discuss whether $\cC$ contains a triple of elements whose
product is $1$ which generate $G$. Moreover, we determine which
elements in $G$ can be written as a product of two conjugate
elements that generate $G$.
\end{abstract}


\section{Introduction}
There is a great interest in expansion properties of conjugacy
classes in finite simple (non-abelian) groups. Thompson conjectured
that every finite simple group $G$ has a conjugacy class $\cC$ such
that $\cC^2 = G$. This conjecture was verified for many families of
finite simple groups, including the alternating groups and the
sporadic groups, and Ellers and Gordeev~\cite{EG} proved this
conjecture for all finite simple groups of Lie type defined over
fields with more than $8$ elements, but nevertheless it is still
very much open today. On the other hand, the expansion of small
enough conjugacy classes $\cC$ in sufficiently large finite simple
groups, has been investigated by Schul~\cite{Sch}.
Moreover, Guralnick and Malle~\cite{GM} showed that every finite
simple group $G$ has a conjugacy class $\cC$ that contains a triple
of elements which have product $1$ and generate $G$.

\medskip

In this paper we consider \emph{all} the conjugacy classes in the
group $G=\PSL_2(q)$, and extend previous results of Guralnick and
Malle~\cite[Lemma 3.14, Lemma 3.15 and Theorem 7.1]{GM}.

\begin{thm}\label{thm.generate}
Let $G=\PSL_2(q)$ where $q>3$ and let $\cC$ be a non-trivial
$G$-conjugacy class. Then $\cC$ contains elements $x,y$ such that
$\langle x,y \rangle = G$, with only the following exceptions:
\begin{itemize}
\item $\cC$ is a conjugacy class of an element of order $2$;
\item $q=9$ and $\cC$ is a conjugacy class of a unipotent element (of order
$3$).
\end{itemize}
In addition, $\cC$ contains three elements $x,y,z$ such that $xyz=1$
and $\langle x,y \rangle = G$ if and only if one of the following
holds:
\begin{itemize}
\item $\cC$ is a conjugacy class of a semisimple element of a
\emph{$q$-minimal} order greater than $3$;
\item $q>3$ is prime and $\cC$ is a conjugacy class of a unipotent
element.
\end{itemize}
\end{thm}

\begin{defn}\label{defn.q.min.order}
Let $q=p^e$ be a prime power and let $n>1$ be an integer. Then $n$
is called a \emph{$q$-minimal} order if $$e=\min \bigl\{f>0 \,:\,
p^f \equiv \pm 1 \bmod{(\gcd(2,n)\cdot n)}\bigr\}.$$ (Namely,
$\PSL_2(p^e)$ contains an element of order $n$, but no
$\PSL_2(p^f)$, $f<e$ contains such an element.)
\end{defn}

Note that similar results appear in~\cite{LR1,LR2,Mac} regarding
generation properties of $\PSL_2(q)$ by a triple of elements
$(x,y,z)$ with product $1$ of prescribed orders (see
also~\cite{Ga,Mar}).

\begin{thm}\label{thm.PSL.even}
Let $G=\PSL_2(q)$ where $q>2$ is even.

Let $\cC$ be the $G$-conjugacy class of an element $x$.
\begin{enumerate}
\item If $x$ is a semisimple element whose order divides $q-1$ then
$\cC^2=G$.
\item If $x$ is a semisimple element whose order divides
$q+1$ then $\cC^2=G \setminus \{unipotents\}$.
\item If $x$ is a unipotent element then $\cC^2=G$.
\end{enumerate}
In addition, only a semisimple element in $G$ can be written as a
product of two $G$-conjugate (semisimple) elements that generate
$G$.
\end{thm}

\begin{thm}\label{thm.PSL.odd}
Let $G=\PSL_2(q)$ where $q>3$ is odd.

Let $\cC$ be the $G$-conjugacy class of an element $x$.
\begin{enumerate}
\item If $x$ is a semisimple element whose order is greater than $2$ then $\cC^2=G$.

\item If $x$ is an element of order $2$ then
$\cC^2=\begin{cases}
G & if \ q \equiv 1 \bmod 4 \\
G\setminus \{unipotents\} & if \ q \equiv 3 \bmod 4
\end{cases}.$

\item If $x$ is a unipotent element then
\begin{equation*}
\cC^2=\begin{cases}
\{unipotents\} \cup \{semisimples\ of\ q-good\ orders\} \cup \{1\} & if \ q \equiv 1 \bmod 4 \\
\{unipotents\} \cup \{semisimples\ of\ q-good\ orders\} & if \ q
\equiv 3 \bmod 4
\end{cases}.
\end{equation*}
\end{enumerate}
In addition, any non-trivial element in $G$ can be written as a
product of two $G$-conjugate semisimple elements that generate $G$.
And moreover, a semisimple element can be written as a product of
two $G$-conjugate unipotents that generate $G$ if and only if its
order is \emph{$q$-minimal} and \emph{$q$-good}, where $q \neq 9$.
\end{thm}

\begin{defn}\label{defn.G.good.order}
Let $q=p^e$ be a prime power and let $n>1$ be an integer. Then $n$
is called a \emph{$q$-good} order if one of the following holds:
\begin{itemize}
\item $n$ is odd and divides either $q-1$ or $q+1$;
\item $n$ is even and $4n$ divides either $q-1$ or $q+1$.
\end{itemize}
\end{defn}

\begin{cor}\label{cor.all.sizes}
Let $G=\PSL_2(q)$ and let $\cC$ be the $G$-conjugacy class of an
element $x$.
\begin{enumerate}
\item If $q$ is even then
\[
\frac{|\cC|}{|G|^{2/3}} \rightarrow 1 \ and \ \frac{|\cC^2|}{|G|}
\rightarrow 1 \ as \ q \rightarrow \infty.
\]
\item If $q$ is odd and $x$ is a semisimple element whose order is greater than
$2$ then
\[
\frac{|\cC|}{|G|^{2/3}} \rightarrow \sqrt[3]{4} \ and \
\frac{|\cC^2|}{|G|} \rightarrow 1 \ as \ q \rightarrow \infty.
\]
\item If $q$ is odd and $x$ is an element of order $2$ then
\[
\frac{|\cC|}{|G|^{2/3}} \rightarrow \frac{1}{\sqrt[3]{2}} \ and \
\frac{|\cC^2|}{|G|} \rightarrow 1 \ as \ q \rightarrow \infty.
\]
\item If $q$ is odd and $x$ is a unipotent element then
\[
\frac{|\cC|}{|G|^{2/3}} \rightarrow \frac{1}{\sqrt[3]{2}} \ and \
\frac{|\cC^2|}{|G|} \rightarrow \frac{3}{4} \ as \ q \rightarrow
\infty.
\]
\end{enumerate}
\end{cor}

\begin{example}
For any prime power $q < 30$ the following table presents \emph{all}
the \emph{$q$-minimal} orders, divided according to whether they are
\emph{$q$-good} or not.

\begin{table}[h]
\begin{tabular} {|c|c|c|c|}
\hline
$q$ & order of a & $q$-minimal orders &  $q$-minimal orders \\
& unipotent & which are $q$-good & which are not $q$-good \\
\hline \hline

$2$ & $2$ & $3$ & -- \\
\hline

$3$ & $3$ & -- & $2$ \\
\hline

$4$ & $2$ & $5$ & -- \\
\hline

$5$ & $5$ & $3$ & $2$ \\
\hline

$7$ & $7$ & $2,3$ & $4$\\
\hline

$8$ & $2$ & $7, 9$ & -- \\
\hline

$9$ & $3$ & $5$ & $4$\\
\hline

$11$ & $11$ & $3, 5$ &  $2, 6$\\
\hline

$13$ & $13$ & $3, 7$ & $2, 6$ \\
\hline

$16$ & $2$ & $15,17$ & -- \\
\hline

$17$ & $17$ & $2,3,4,9$ & $8$ \\
\hline

$19$ & $19$ & $3,5,9$ & $2,10$ \\
\hline

$23$ & $23$ & $2,3,6,11$ & $4,12$ \\
\hline

$25$ & $5$ & $6, 13$ & $4, 12$ \\
\hline

$27$ & $3$ & $7,13$ & $14$ \\
\hline

$29$ & $29$ & $3,5,7,15$ & $2,14$ \\
\hline
\end{tabular}
\caption{Orders of elements in $\PSL_2(q)$} \label{table.ord.PSL}
\end{table}
\end{example}

\section{Preliminaries -- basic properties of $\PSL_2(q)$}

\subsection{Elements and conjugacy classes}
(see~\cite{Di,Do,Go,Su}). Let $q = p^e$, where $p$ is a prime number
and $e \geq 1$. Recall that the order of $G=\PSL_2(q)$ is
$q(q^2-1)/d$, where
$$d=\gcd(2,q-1) = \begin{cases} 1 \ & \text{ if } q \text{ is  even} \\
2 \ & \text{ if } q \text{ is  odd}
\end{cases}.$$

One can classify the elements of $\PSL_2(q)$ according to the Jordan
form of their pre-image in $\SL_2(q)$. The following Table
\ref{table.elm.PSL} lists the three types of elements, according to
whether the characteristic polynomial $P(\la):=\la^2 - \al \la +1$
of the matrix $A \in \SL_2(q)$ (where $\al=\tr(A)$) has $0$, $1$ or
$2$ distinct roots in $\BF_q$.

We denote $\ka(\al)= \begin{cases}1 \ & \text{ if } \al \ne 0 \\
2 \ & \text{ if } \al=0 \end{cases}.$

\begin{table}[h]
\begin{tabular} {|c|c|c|c|c|}
\hline
type & roots of $P(\la)$ & Jordan form in $\SL_2(\overline{\mathbb{F}}_p)$ & order & $G$-conjugacy classes  \\
\hline \hline

unipotent & $1$ root & $\begin{pmatrix} \pm 1 & 1 \\ 0 & \pm 1 \end{pmatrix}$ & $p$ & $d$ classes in $G$ \\
 & & $\al=\pm 2$ &  & each of size $\frac{q^2-1}{d}$ \\
\hline

semisimple & $2$ roots & $\begin{pmatrix} a & 0 \\ 0 & a^{-1}\end{pmatrix}$ & divides & for each $\al$: \\
split      &           & where $a \in \mathbb{F}_q^*$ & $\frac{q-1}{d}$ &  one conjugacy class\\
      &           & and $a+a^{-1}=\al$ &  & of size $\frac{q(q+1)}{\ka(\al)}$\\
\hline

semisimple & no roots & $\begin{pmatrix} a & 0 \\ 0 & a^q\end{pmatrix}$ & divides & for each $\al$: \\
non-split  &          & where $a \in \mathbb{F}_{q^2}^* \setminus \mathbb{F}_q^*$ & $\frac{q+1}{d}$ & one conjugacy class \\
          &          & $a^{q+1}=1$ and $a + a^q = \al$ & & of size $\frac{q(q-1)}{\ka(\al)}$\\
\hline
\end{tabular}
\caption{Elements in $\PSL_2(q)$}
\label{table.elm.PSL}
\end{table}

\subsection{Subgroups}
Table~\ref{table.subgroups} specifies all the subgroups of
$G=\PSL_2(q)$ up to isomorphism following~\cite[Theorems 6.25 and
6.26]{Su}.

\begin{table}[h]
\begin{tabular} {|c|c|c|}
\hline
type &  maximal order & conditions \\
\hline \hline

$p$-group & $q$ & -- \\
\hline

Frobenius (Borel) & $q(q-1)/d$ & -- \\
\hline

cyclic (split) & $(q-1)/d$ & -- \\
\hline

dihedral (split) & $2(q-1)/d$ & -- \\
\hline

cyclic (non-split) & $(q+1)/d$ & -- \\
\hline

dihedral (non-split) & $2(q+1)/d$ & -- \\
\hline

$\PSL_2(q_1)$ & -- & $q=q_1^m$ $(m \in \mathbb{N})$\\
\hline

$\PGL_2(q_1)$ & -- & $q$ is odd, $q=q_1^{2m}$ $(m \in \mathbb{N})$ \\
\hline

$A_4$ & 12 &  $q$ is odd; or $q=2^e$, $e$ even \\
\hline

$S_4$ & 24 & $q^2 \equiv 1 \bmod {16}$ \\
\hline

$A_5$ & 60 & $p=5$ or $q^2 \equiv 1 \bmod {5}$ \\
\hline
\end{tabular}
\caption{Subgroups of $\PSL_2(q)$} \label{table.subgroups}
\end{table}

These subgroups can be divided into the following three classes,
following Macbeath~\cite{Mac}. The subgroups isomorphic to
$\PSL_2(q_1)$ or $\PGL_2(q_1)$ are usually called \emph{subfield}
subgroups (since $\mathbb{F}_{q_1}$ is a subfield of
$\mathbb{F}_q$). Since $A_4$, $S_4$, $A_5$ and dihedral groups
correspond to the finite triangle groups, that is, triangle groups
$T_{r,s,t}$ such that $1/r+1/s+1/t>1$, we will call them
\emph{small} subgroups. For convenience we will refer to the other
subgroups, namely subgroups of the Borel and cyclic subgroups, as
\emph{structural} subgroups.

\medskip
Macbeath~\cite{Mac} classified the pairs of elements in $\PSL_2(q)$
in a way which makes it easy to decide what kind of subgroup they
generate.

\begin{thm}\cite[Theorem 1]{Mac}. \label{thm.Mac}
For every $\al,\be,\ga \in \BF_q$ there exist three matrices $A,B,C
\in \SL_2(q)$ satisfying $\tr(A)=\al$, $\tr(B)=\be$, $\tr(C)=\ga$
and $ABC=I$.
\end{thm}

\begin{defn}\label{defn.singular}
A triple $(\al,\be,\ga) \in \BF_q^3$ is called \emph{singular} if
$\al^2+\be^2+\ga^2-\al\be\ga - 4 = 0$.
\end{defn}

\begin{thm}\cite[Theorem 2]{Mac}. \label{thm.Mac.singular}
Let $(A,B,C) \in \SL_2(q)^3$ be a triple of matrices satisfying
$\tr(A)=\al$, $\tr(B)=\be$, $\tr(C)=\ga$ and $ABC=I$. Then
$(\al,\be,\ga) \in \BF_q^3$ is \emph{singular} if and only if the
group generated by the images of $A$ and $B$ is a \emph{structural}
subgroup of $\PSL_2(q)$.
\end{thm}

\subsection{Orders and traces}
For an integer $n>1$ we denote
\begin{equation}
\Tr_q(n) = \{\al \in \mathbb{F}_q :\, \al=\tr(A), A \in \SL_2(q),
|\bar{A}|=n\}.
\end{equation}
It is easy to see from Table~\ref{table.elm.PSL} that for any prime
power $q$, $\Tr_q(2)=\{0\}$, $\Tr_q(3)=\{\pm 1\}$, and for any odd
$q=p^e$, $\Tr_q(p)=\{\pm 2\}$. Moreover, when $q$ is odd, then $\al
\in \Tr_q(n)$ if and only if $-\al \in \Tr_q(n)$. In fact, for any
prime power $q$ and integer $n>1$, $\Tr_q(n)$ can be effectively
computed as follows.

\begin{prop}\label{prop.Tr.q.n}
Denote by $\Pri_q(n)$ the set of primitive roots of unity of
order $n$ in $\mathbb{F}_q$.
\begin{itemize}
\item Let $q=2^e$ for some positive integer $e$ and let $n>1$ be an integer, then
\begin{equation}
\Tr_q(n) = \begin{cases} \{0\} & if \ n=2 \\
\{a+a^{-1}: a \in \Pri_q(n)\} & if \ n \ divides \ q-1 \\
\{b+b^q: b \in \Pri_{q^2}(n)\} & if \ n \ divides \ q+1 \\
\emptyset & otherwise
\end{cases}
\end{equation}

\item Let $q=p^e$ for some odd prime $p$ and some positive integer $e$ and let $n>1$ be an integer, then
\begin{equation}
\Tr_q(n) = \begin{cases} \{\pm 2\} & if \ n=p \\
\{\pm (a+a^{-1}): a \in \Pri_q(2n)\} & if \ n \ divides \ \frac{q-1}{2} \\
\{\pm (b+b^q): b \in \Pri_{q^2}(2n)\} & if \ n \ divides \ \frac{q+1}{2} \\
\emptyset & otherwise
\end{cases}
\end{equation}
\end{itemize}
\end{prop}

\begin{proof} We prove the case where $q$ is odd and $n$ divides $(q-1)/2$.
The other cases are similar.

Assume first that $n$ is even. Let $a$ be a primitive root of unity
of order $2n$. Then $-a$ is also a primitive root of unity of order
$2n$, and $(-a)^n = a^n = -1$. Thus the matrices
\[
    A = \begin{pmatrix}
    a & 0 \\ 0 & a^{-1}
    \end{pmatrix}
    \quad \text{and} \quad
    -A = \begin{pmatrix}
    -a & 0 \\ 0 & -a^{-1}
    \end{pmatrix}
\]
both reduce to the same element $\bar{A} \in G$ of order $n$. Hence,
$a+a^{-1}$ and $-(a+a^{-1})$ both belong to $\Tr_q(n)$. $\Tr_q(n)$
contains only the elements of the claimed form, since it suffices to
consider only the conjugacy classes of elements of $G$, described in
Table~\ref{table.elm.PSL}.

Now assume that $n$ is odd. Then $a$ is a primitive root of unity of
order $n$ if and only if $-a$ is a primitive root of unity of order
$2n$. In this case, $(-a)^n=-a^n=-1$. Thus the matrices
\[
    A = \begin{pmatrix}
    a & 0 \\ 0 & a^{-1}
    \end{pmatrix}
    \quad \text{and} \quad
    -A = \begin{pmatrix}
    -a & 0 \\ 0 & -a^{-1}
    \end{pmatrix}
\]
both reduce to the same element $\bar{A} \in G$ of order $n$. Hence,
$a+a^{-1}$ and $-(a+a^{-1})$ both belong to $\Tr_q(n)$. Again,
considering the conjugacy classes appearing in
Table~\ref{table.elm.PSL} shows that $\Tr_q(n)$ contains only the
elements of the claimed form.
\end{proof}

\begin{prop}\label{prop.good.G.order}
Assume that $q=p^e$ is odd. Let $C \in \SL_2(q)$, and denote
$\ga=\tr(C)$ and $t=|\bar{C}|$. Assume that $\ga \ne \pm 2$ (or
equivalently, $t \ne p$). Then $t$ is a \emph{$q$-good} order if and
only if one of $2+\ga$, $2-\ga$ is a square in $\mathbb{F}_q$.
\end{prop}
\begin{proof}
As $t \ne p$ then $t$ divides either $(q-1)/2$ or $(q+1)/2$.

If $t$ divides $(q-1)/2$ then $\ga=a+a^{-1}$ or $\ga=-(a+a^{-1})$,
for some primitive root of unity $a$ of order $2t$ in $\mathbb{F}_q$
(see Proposition~\ref{prop.Tr.q.n}). Hence,
\[
    \{2+\ga, 2-\ga\} = \{a+2+a^{-1},(-a)+2+(-a)^{-1}\}.
\]
Therefore, $2+\ga$ or $2-\ga$ is a square in $\mathbb{F}_q$ if and
only if $a=c^2$ or $-a=c^2$ for some $c \in \mathbb{F}_q$. Indeed,
$a+2+a^{-1}$ is a square if and only if $(a+1)^2/a$ is a square if
and only if $a$ is a square in $\mathbb{F}_q$, and similarly for
$(-a)+2+(-a)^{-1}$.

Now, one the following necessarily holds:
\begin{itemize}
\item If $t$ is even then $-a$ is also a primitive root of unity
of order $2t$. Hence, $a$ is a square in $\mathbb{F}_q$ if and only
if $-a$ is a square is $\mathbb{F}_q$. In addition, $a$ is a square
in $\mathbb{F}_q$ if and only if $\mathbb{F}_q$ contains a primitive
root of unity of order $4t$, namely, if and only if $4t$ divides
$q-1$.

\item If $t$ is odd and $q \equiv 1 \bmod 4$ then $4t$ divides
$q-1$, and so $\mathbb{F}_q$ contains a primitive root of unity of
order $4t$. Thus $a$, which is a primitive root of unity of order
$2t$, is a square in $\mathbb{F}_q$, as required.

\item If $t$ is odd and $q \equiv 3 \bmod 4$ then $\mathbb{F}_q$
contains a primitive root of unity of order $2t$ but does not
contain a primitive root of unity of order $4t$, and so, $a$ is a
non-square in $\mathbb{F}_q$. However, $-a$ is a primitive root of
unity of order $t$, and so, it is necessarily a square in
$\mathbb{F}_q$, as required.
\end{itemize}
In conclusion, $a=c^2$ or $-a=c^2$ for some $c \in \mathbb{F}_q$ if
and only if either $t$ is odd and divides $q-1$ or $t$ is even and
$4t$ divides $q-1$.

\medskip

If $t$ divides $(q+1)/2$ then $\ga=a+a^q$ or $\ga=-(a+a^q)$, for
some primitive root of unity $a$ of order $2t$ in $\mathbb{F}_{q^2}$
(see Proposition~\ref{prop.Tr.q.n}). Hence,
\[
    \{2+\ga,2-\ga\} = \{a+2+a^q,(-a)+2+(-a)^q\}.
\]
Therefore, $2+\ga$ or $2-\ga$ is a square in $\mathbb{F}_q$ if and
only if $a=c^2$ or $-a=c^2$ for some $c \in \mathbb{F}_{q^2}$
satisfying $c^{q+1}=1$.

Now, one the following necessarily holds:
\begin{itemize}
\item If $t$ is even then $-a$ is also a primitive root of unity
of order $2t$. Hence, $a=c^2$ for some $c \in \mathbb{F}_{q^2}$
satisfying $c^{q+1}=1$ if and only if $-a=b^2$ for some $b \in
\mathbb{F}_{q^2}$ satisfying $b^{q+1}=1$. This is equivalent to the
condition that $4t$ divides $q+1$.

\item If $t$ is odd and $q \equiv 3 \bmod 4$ then $\mathbb{F}_{q^2}$
contains a primitive root of unity $b$ of order $4t$ satisfying
$b^{q+1}=1$. Hence, $a=c^2$ for some $c \in \mathbb{F}_{q^2}$
satisfying $c^{q+1}=1$, as required.

\item If $t$ is odd and $q \equiv 1 \bmod 4$ then $\mathbb{F}_{q^2}$
does not contain a primitive root of unity $c$ of order $4t$
satisfying $c^{q+1}=1$. However, in this case, $-a=b^2$ for some $b
\in \mathbb{F}_{q^2}$ satisfying $b^{q+1}=1$, as required.
\end{itemize}
In conclusion, $a=c^2$ or $-a=c^2$ for some $c \in \mathbb{F}_{q^2}$
satisfying $c^{q+1}=1$ if and only if either $t$ is odd and divides
$q+1$ or $t$ is even and $4t$ divides $q+1$.
\end{proof}

\subsection{Number of elements}\label{sect.num}
\begin{defn}
Let $q$ be a prime power and let $\al \in \BF_q$.

If $\al=\pm 2$ we say that $\al$ is a \emph{unipotent} trace. If
$\al$ is a trace of a semisimple split (respectively, non-split)
matrix in $\SL_2(q)$ we say that $\al$ is \emph{split}
(respectively, \emph{non-split}).

Assume that $q$ is odd and $\al \ne \pm 2$. If at least one of
$2+\al$, $2-\al$ is a square in $\BF_q$ we say that $\al$ is a
\emph{good} trace. If none of $2+\al$, $2-\al$ is a square in
$\BF_q$ we say that $\al$ is a \emph{bad} trace.
\end{defn}

\begin{lemma}\label{lem.num.tr} ${}$
\begin{enumerate}
\item If $q$ is even then $\BF_q$ contains one unipotent,
$(q-2)/2$ split and $q/2$ non-split traces.
\item If $q$ is odd then $\BF_q$ contains $2$ unipotent,
$(q-3)/2$ split and $(q-1)/2$ non-split traces.
\item If $q \equiv 1 \bmod 4$ then $\BF_q$ contains $(q-1)/4$ bad
traces.
\item If $q \equiv 3 \bmod 4$ then $\BF_q$ contains $(q+1)/4$ bad
traces.
\end{enumerate}
\end{lemma}
\begin{proof}
(1) and (2) are due to~\cite[Lemma 2]{Mac}.


(3) If $q \equiv 1 \bmod 4$ then $(q-1)/2$ is even, and so all the
bad traces are necessarily split. Hence, there are
$\frac{1}{2}\bigl(\frac{1}{2}(q-3)+1\bigr)$ bad traces.

(4) If $q \equiv 3 \bmod 4$ then $(q+1)/2$ is even, and so all the
bad traces are non-split. Hence, there are
$\frac{1}{2}\bigl(\frac{1}{2}(q-1)+1\bigr)$ bad traces.
\end{proof}

Now, the following corollaries easily follow from
Table~\ref{table.elm.PSL}, Proposition~\ref{prop.good.G.order} and
Lemma~\ref{lem.num.tr}.

\begin{cor}\label{cor.num.elm.even}
If $q$ is even then $G=\PSL_2(q)$ contains:
\begin{itemize}
\item $q^2-1$ unipotent elements,
\item $\frac{1}{2}q(q+1)(q-2)$ semisimple split elements,
\item $\frac{1}{2}q^2(q-1)$ semisimple non-split elements.
\end{itemize}
\end{cor}

\begin{cor}\label{cor.num.elm.odd}
If $q$ is even then $G=\PSL_2(q)$ contains:
\begin{itemize}
\item $q^2-1$ unipotent elements,
\item $\frac{1}{4}q(q+1)(q-3)$ semisimple split elements,
\item $\frac{1}{4}q(q-1)^2$ semisimple non-split elements.
\item $\frac{1}{8}q(q^2-1)=\frac{|G|}{4}$ semisimple
elements whose order is \emph{not $q$-good}.
\end{itemize}
\end{cor}

\subsection{Unipotent elements when $q$ is odd}\label{sect.unip.elm}
Consider the following matrices in $G_0=\SL_2(q)$ (where $q$ is
odd):
\[
    U_1 = \begin{pmatrix} 1 & 1 \\ 0 & 1 \end{pmatrix}, \quad
    U_{-1} = \begin{pmatrix} -1 & 1 \\ 0 & -1 \end{pmatrix} = -U_1^{-1},
\]
\[
    U'_1 = XU_1X^{-1} = \begin{pmatrix} 1 & x^2 \\ 0 & 1 \end{pmatrix} \in
    G_0, \quad
    U'_{-1} = XU_{-1}X^{-1} = \begin{pmatrix} -1 & x^2 \\ 0 & -1 \end{pmatrix} \in
    G_0,
\]
where $x \in \mathbb{F}_{q^2} \setminus \mathbb{F}_q$ satisfies that
$x^2 \in \mathbb{F}_q$ and $X=\begin{pmatrix} x & 0 \\ 0 & x^{-1}
\end{pmatrix} \in \SL_2(q^2)$.


\begin{prop}\label{prop.unip.G0}
Let $G_0=\SL_2(q)$ when $q$ is odd. Then, for any $A \in G_0$,
$XAX^{-1} \in G_0$. Moreover,
\begin{itemize}
\item If $A \ne I$ and $\tr(A)=2$ then $A$ is $G_0$-conjugate to either $U_1$ or $U'_1$.
\item If $A \ne -I$ and $\tr(A)=-2$ then $A$ is $G_0$-conjugate to either $U_{-1}$ or $U'_{-1}$.
\end{itemize}
\end{prop}
\begin{proof}
Indeed, if $A=\begin{pmatrix} a & b \\ c & d \end{pmatrix} \in G_0$
then $XAX^{-1} = \begin{pmatrix} a & bx^2 \\ cx^{-2} & d
\end{pmatrix} \in G_0$. Moreover, $U_1$ is not $G_0$-conjugate to
$U'_1=XU_1X^{-1}$, since $x^2$ is not a square of some element in
$\mathbb{F}_q$. Hence, any $I \ne A \in G_0$ with $\tr(A)=2$ is
$G_0$-conjugate to either $U_1$ or $U'_1$ (see
Table~\ref{table.elm.PSL}).
\end{proof}

\begin{cor}\label{corr.unip.G}
Let $G=\PSL_2(q)$ when $q$ is odd. Then, for any $\bar{A} \in G$,
$\bar{X}\bar{A}\bar{X}^{-1} \in G$. If moreover, $\bar{A}$ is
unipotent then it is $G$-conjugate to either $\bar{U}_1$ or
$\bar{U'}_1=\bar{X}\bar{U}_1\bar{X}^{-1}$. In addition,
\begin{itemize}
\item If $q \equiv 1 \bmod 4$ then $\bar{U}_1$ and $\bar{U}_{-1}$ are $G$-conjugate.
\item If $q \equiv 3 \bmod 4$ then $\bar{U}_1$ and $\bar{U'}_{-1}$ are $G$-conjugate.
\end{itemize}
\end{cor}

\section{Proof of the main results}
\subsection*{Proof of Theorem~\ref{thm.generate}}
The proof follows from the following
Proposition~\ref{prop.gen.semisimple},
Proposition~\ref{prop.gen.semisimple.triple},
Proposition~\ref{prop.gen.triple.ppp}, Remark~\ref{rem.22.333},
Proposition~\ref{prop.good.unip.triple} and Remark~\ref{rem.9.335}.

\subsection*{Proof of Theorem~\ref{thm.PSL.even}}
The proof follows from the following
Proposition~\ref{prop.gen.ss.ss}, Proposition~\ref{prop.semisimple},
Proposition~\ref{prop.even.unip.ss},
Proposition~\ref{prop.split.unip} and
Proposition~\ref{prop.non.split.even}.

\subsection*{Proof of Theorem~\ref{thm.PSL.odd}}
The proof follows from the following
Proposition~\ref{prop.gen.ss.ss}, Proposition~\ref{prop.odd.unip},
Proposition~\ref{prop.semisimple},
Proposition~\ref{prop.split.unip},
Proposition~\ref{prop.non.split.unip}, Corollary~\ref{cor.conj.inv},
Corollary~\ref{cor.odd.unip.ss} and
Proposition~\ref{prop.good.unip.triple}.

\subsection*{Proof of Corollary~\ref{cor.all.sizes}}
The proof follows from Theorem~\ref{thm.PSL.even},
Theorem~\ref{thm.PSL.odd}, Table~\ref{table.elm.PSL} and
\S\ref{sect.num}.

\begin{enumerate}
\item Let $G=\PSL_2(q)$ when $q$ is even, then $|G|=q(q^2-1)$.
\begin{itemize}
\item If $\cC$ is a conjugacy class of a semisimple split element then
\[
    |\cC|=q(q+1)\ and\ |\cC^2|=q(q^2-1).
\]
\item If $\cC$ is a conjugacy class of a semisimple non-split element then
\[
    |\cC|=q(q-1)\ and\ |\cC^2|=q(q^2-1)-(q^2-1)=(q-1)(q^2-1).
\]
\item If $\cC$ is a conjugacy class of a unipotent element then
\[
    |\cC|=q^2-1\ and\ |\cC^2|=q(q^2-1).
\]
\end{itemize}
\medskip
Now let $G=\PSL_2(q)$ when $q$ is odd, then $|G|=q(q^2-1)/2$.

\item If $\cC$ is a conjugacy class of semisimple element whose order is greater than
2 then
\[
    |\cC|=q(q\pm 1)\ and\ |\cC^2|=q(q^2-1)/2.
\]

\item If $\cC$ is a conjugacy class of semisimple element $x$ of order 2 then
\begin{itemize}
\item If $q \equiv 1 \bmod 4$ then $x$ is split and so
\[
    |\cC|=q(q+1)/2 \ and\ |\cC^2|=q(q^2-1)/2.
\]
\item If $q \equiv 3 \bmod 4$ then $x$ is non-split and so
\[
    |\cC|=q(q-1)/2 \ and\ |\cC^2|=q(q^2-1)/2-(q^2-1) = (q-2)(q^2-1)/2.
\]
\end{itemize}

\item If $\cC$ is a conjugacy class of a unipotent element, then by
Corollary~\ref{cor.num.elm.odd},
\[
    |\cC|=(q^2-1)/2\ and\ |\cC^2|=q(q^2-1)/2 - q(q^2-1)/8 -\eps = 3q(q^2-1)/8
    -\eps,
\]
where $\eps = \begin{cases} 1 & if\ q \equiv 1 \bmod 4 \\
0 & if\ q \equiv 3 \bmod 4.
\end{cases}$
\end{enumerate}

\subsection{Generation properties of $\cC$}\label{sect.gen}
\begin{prop}\label{prop.gen.semisimple}
Let $G=\PSL_2(q)$ when $q>3$ and let $\cC$ be a conjugacy class of a
semisimple element $s$ of order greater than $2$. Then $\cC$
contains two elements $x,y$ that generate $G$.
\end{prop}
\begin{proof}
Let $s \in G$ be a semisimple element and let $\cC$ be the conjugacy
class of $s$. Let $S \in G_0=\SL_2(q)$ be the preimage of $s$.
Denote the order of $s$ by $n$ and $\al = \tr(S)$, thus $\al \notin
\{0, \pm 2\}$. Recall that $d=\gcd(2,q-1)$.

If $q\neq 5,7$ take some $\ga \in \Tr_q((q+1)/d)$. If
$(\al,\al,\ga)$ is \emph{singular} then $(2-\ga)(\al^2-\ga-2)=0$ and
so $\ga=\al^2-2$. Thus, we may replace $\ga$ by $-\ga$ (or by
another $\ga \in \Tr_q((q+1)/d)$) to obtain a non-singular triple
$(\al,\al,\ga)$. When $q=9$ one can moreover make sure that
$(\al,\al,\ga)$ does not correspond to a triple of matrices
$(A,B,C)$ with $ABC=I$ satisfying $\langle \bar A, \bar B \rangle
\cong A_5$ (see~\cite{LR2}). When $q=5$ or $7$ take $\ga=-2$, and
then the triple $(\al,\al,-2)$ is also not singular.

By Theorem~\ref{thm.Mac}, there exists a triple of matrices
$(A,B,C)$ such that $ABC=I$, $\tr(A)=\tr(B)=\al$ and $\tr(C)=\ga$.
Moreover, by Theorem~\ref{thm.Mac.singular}, $\langle \bar A, \bar B
\rangle$ is not a structural subgroup of $G$. By considering the
possible subgroups of $G$ detailed in Table~\ref{table.subgroups} we
see that the only non-structural subgroup containing $\bar C$ is $G$
itself, and hence $\langle \bar A, \bar B \rangle = G$.
\end{proof}

\begin{prop}\label{prop.gen.ss.ss}
Let $G=\PSL_2(q)$ when $q>3$.
\begin{itemize}
\item If $q$ is odd then any non-trivial element in $G$ can be written
as a product of two $G$-conjugate semisimple elements that generate
$G$;
\item If $q$ is even then only a semisimple element in $G$ can be written as a product of two $G$-conjugate
(semisimple) elements that generate $G$.
\end{itemize}
\end{prop}
\begin{proof}
Assume that $q$ is even. Let $Z\in G$ and denote $\ga = \tr(Z)$.
Assume that $Z=XY$ where $X$ is $G$-conjugate to $Y$, and let
$\al=\tr(X)=\tr(Y)$. If $Z$ is unipotent then $\ga=0$. Since
$(\al,\al,0)$ is \emph{singular} then by
Theorem~\ref{thm.Mac.singular} the subgroup $\langle X,Y \rangle$ is
a structural subgroup of $G$, and so $X$ and $Y$ cannot generate
$G$. If $Z$ is semisimple then $\ga\neq 0$, and we can choose some
$\al \in \Tr_q(q+1)$ such that $\ga \neq \al^2$ and then
$(\al,\al,\ga)$ is not singular.

Assume that $q$ is odd. Let $1 \neq z\in G$ and let $Z \in
G_0=\SL_2(q)$ be the preimage of $z$. Denote $\ga = \tr(Z)$. If $q
\neq 5,7$ take some $\al \in \Tr_q((q+1)/2)$. If $(\al,\al,\ga)$ is
\emph{singular} then $(2-\ga)(\al^2-\ga-2)=0$ and so $\ga=-2$ or
$\ga=\al^2-2$. Thus, we may replace $\ga$ by $-\ga$ to obtain a
non-singular triple $(\al,\al,\ga)$. When $q=9$ one can moreover
make sure that $(\al,\al,\ga)$ does not correspond to a triple of
matrices $(A,B,C)$ with $ABC=I$ satisfying $\langle \bar A, \bar B
\rangle \cong A_5$ (see~\cite{LR2}). When $q=5$ or $7$ take $\al
=-2$ (and $\ga \ne 2$) and then the triple $(-2,-2,\ga)$ is also not
singular.

By Theorem~\ref{thm.Mac}, there exists a triple of matrices
$(A,B,C)$ such that $ABC=I$, $\tr(A)=\tr(B)=\al$ and $\tr(C)=\ga$.
Moreover, by Theorem~\ref{thm.Mac.singular}, $\langle \bar A, \bar B
\rangle$ is not a structural subgroup of $G$. By considering the
possible subgroups of $G$ detailed in Table~\ref{table.subgroups} we
see that the only non-structural subgroup containing $\bar A$ is $G$
itself, and hence $\langle \bar A, \bar B \rangle = G$.
\end{proof}

\begin{prop}\label{prop.gen.semisimple.triple}
Let $G=\PSL_2(q)$ when $q>3$ and let $\cC$ be a conjugacy class of a
semisimple element $s$. Then $\cC$ contains three elements $x,y,z$
such that $xyz=1$ and $\langle x,y \rangle =G$ if and only if the
order of $s$ is \emph{$q$-minimal} and greater than $3$.
\end{prop}
\begin{proof}
Let $s \in G$ be a semisimple element and let $\cC$ be the conjugacy
class of $s$. Let $S \in G_0=\SL_2(q)$ be the preimage of $s$.
Denote the order of $s$ by $n$ and $\al = \tr(S)$.

If $x,y,z$ are three elements of order $2$ (respectively, $3$) such
that $xyz=1$, then any finite group generated by $x$ and $y$ is
necessarily abelian (respectively, solvable) and so $x$ and $y$
cannot generate $G$ (see Remark~\ref{rem.22.333}).

If $n$ is not $q$-minimal then $\al$ belongs to some proper subfield
$\BF_{q_1}$ of $\BF_q$. In this case, for any three elements
$x,y,z\in \cC$ satisfying $xyz=1$, the subgroup $\langle x,y
\rangle$ is necessarily contained in a proper subgroup of $G$
isomorphic to $\PSL_2(q_1)$ (see~\cite{Mac} and~\cite{LR1}).

If $n>3$ is $q$-minimal then $\al \notin \{0,\pm 1,\pm 2\}$ and
hence $(\al-2)^2(\al+1)\ne 0$ implying that the triple
$(\al,\al,\al)$ is non-singular. If $q>5$ and $n=5$ one can moreover
make sure that the triple $(\al,\al,\al)$ does not correspond to a
triple of matrices $(A,B,C)$ with $ABC=I$ satisfying $\langle \bar
A, \bar B \rangle \cong A_5$ (see~\cite{LR2}).

Therefore, by Theorems~\ref{thm.Mac} and~\ref{thm.Mac.singular},
there exists a triple of matrices $(A,B,C)$ such that $ABC=I$,
$\tr(A)=\tr(B)=\tr(C)=\al$ and $\langle \bar A, \bar B \rangle$ is a
subfield subgroup of $G$. By the $q$-minimality of $n$, $s$ does not
belong to any subgroup isomorphic to $\PSL_2(q_1)$ for some proper
subfield $\BF_{q_1} \subset \BF_q$. In addition, if $\langle x,y
\rangle = \PGL_2(q_1)$ and $xyz=1$ then exactly two of $x,y,z$
belong to $\PGL_2(q_1) \setminus \PSL_2(q_1)$ and the third one
belongs to $\PSL_2(q_1)$, so $x,y,z$ cannot all have the same order
$n$. Therefore, $\langle \bar A, \bar B \rangle$ cannot generate a
subgroup isomorphic to $\PGL_2(q_1)$, hence $\langle \bar A, \bar B
\rangle=G$ as needed.
\end{proof}

\begin{prop}\label{prop.gen.triple.ppp}
Let $G=\PSL_2(q)$ where $q$ is a prime power. Then there exist three
$G$-conjugate unipotent elements $A$, $B$ and $C$ such that $ABC=1$
and $\langle A, B \rangle=G$ if and only $q>3$ is prime.
\end{prop}
\begin{proof}
If $A,B,C \in G$ are three elements of prime order $p$ such that
$ABC=1$ then $\langle A,B \rangle$ is contained in a subgroup
isomorphic to $\PSL_2(p)$, implying that necessarily $q=p>3$ is
prime (see~\cite{Mac} and~\cite{LR1}).

Now, choose some $a \in \BF_p \setminus \{0,\pm 2\}$ and consider
the following matrices in $\SL_2(p)$:
$$ M=\begin{pmatrix} a +1 & -a/2-1 \\ 2 & -1
\end{pmatrix}, \quad
K=\begin{pmatrix} a & -1/2 \\ 2 & 0
\end{pmatrix}.
$$
Let $A=U_{-1}\in G_0$ be as in \S\ref{sect.unip.elm},
$B=MU_{-1}M^{-1}$ and $C=KU_{-1}K^{-1}$, then $ABC=I$.

As $\tr(U_{-1})=-2$ and $(-2,-2,-2)$ is a \emph{non-singular}
triple, then by Theorem~\ref{thm.Mac.singular}, the subgroup
generated by the images of $A$ and $B$ is not a structural subgroup
of $G$. By considering the possible subgroups of $G$ detailed in
Table~\ref{table.subgroups} we conclude that $\langle \bar A, \bar B
\rangle = G$, as needed
\end{proof}

\begin{rem}\label{rem.22.333}
Recall that any finite group generated by two non-commuting
involutions is dihedral. In addition, the group presented by
$\langle x,y|\ x^3=y^3=(xy)^3 \rangle$ is an infinite solvable group
(see e.g.~\cite{CM}).
\end{rem}

\subsection{Basic properties of $\cC^2$}
\begin{lemma}\label{lem.conj.tr}
Let $A,B\in G_0=\SL_2(q)$ be two non-central matrices such that
$\tr(A)=\tr(B)$ or $\tr(A)=-\tr(B)$. Let $\cC$ be some conjugacy
class in $G=\PSL_2(q)$. If $\bar A \in \cC^2$ then also $\bar B \in
\cC^2$.
\end{lemma}
\begin{proof}
If $\bar A \in \cC^2$ then $\bar A = z(gzg^{-1})$ for some $z \in
\cC$, $g\in G$.

If $q$ is even, or if $q$ is odd and $\tr(A) \ne \pm 2$, then
necessarily $\bar B=h \bar A h^{-1}$ for some $h \in G$ (see
Table~\ref{table.elm.PSL}), and so $\bar B = (hzh^{-1}) (hgzg^{-1}
h^{-1}) \in \cC^2$.

Assume that $q$ is odd, $z$ is semisimple and $\bar A$ is unipotent.
Let $Z \in G_0$ be the preimage of $z$. Then by
Corollary~\ref{corr.unip.G}, either $\bar B=h \bar A h^{-1}$ or
$\bar B=h \bar X A \bar X^{-1} h^{-1}$ for some $h \in G$ and a
matrix $X \in \SL_2(q^2)$ as in \S\ref{sect.unip.elm}. Moreover by
Proposition~\ref{prop.unip.G0}, $XZX^{-1} \in G_0$, and since $\tr(X
Z X^{-1})=\tr(Z)$ then $\bar X z \bar X^{-1} \in \cC$ (see
Table~\ref{table.elm.PSL}). Similarly, $\bar X gzg^{-1} \bar X^{-1}
\in \cC$.

Hence, either $$\bar B = h \bar A h^{-1} = (hzh^{-1}) (hgzg^{-1}
h^{-1}) \in \cC^2$$ or $$\bar B = h \bar X \bar A \bar X^{-1} h^{-1}
= (h \bar X z \bar X^{-1} h^{-1}) (h \bar X gzg^{-1} \bar X^{-1}
h^{-1}) \in \cC^2.$$

If $q$ is odd and $z$ is unipotent, then the following
Proposition~\ref{prop.odd.unip} shows that $\cC^2$ contains all the
unipotent elements in $G$, as needed.
\end{proof}

\begin{prop}\label{prop.odd.unip}
Let $G=\PSL_2(q)$ when $q$ is odd and let $\cC$ be the $G$-conjugacy
class of a unipotent element $x$. Then $\cC^2$ contains all the
unipotent elements in $G$.

Moreover, $1 \in \cC^2$ if and only if $q \equiv 1 \bmod 4$.
\end{prop}
\begin{proof}
For $q=5$ this claim can be easily verified. When $q>5$ there exist
elements $a,b \in \BF_q$ such that $a,b$ and $a+1$ are squares in
$\BF_q$ but $b+1$ is a non-square in $\BF_q$. Let $U_1$ be as
in~\S\ref{sect.unip.elm}, let $A=\begin{pmatrix}
1 & a \\ 0 & 1\end{pmatrix}$, $B=\begin{pmatrix} 1 & b \\
0 & 1\end{pmatrix}$, then $U_1 A=\begin{pmatrix} 1 & a+1 \\
0 & 1 \end{pmatrix}$ and $U_1 B=\begin{pmatrix} 1 & b+1 \\ 0 & 1
\end{pmatrix}$. Moreover, $A$,$B$ and $U_1A$ are conjugate to $U_1$ in $\SL_2(q)$,
while $U_1B$ is conjugate to $U'_1$ in $\SL_2(q)$ (see
Proposition~\ref{prop.unip.G0}). By Corollary~\ref{corr.unip.G}, any
unipotent element in $G$ is $G$-conjugate to either $\bar U_1$ or to
$\bar U'_1$, as needed.

In addition, by Corollary~\ref{corr.unip.G}, $\bar U_1$ is
$G$-conjugate to $\bar U_1^{-1}$ if and if $q \equiv 1 \bmod 4$.
\end{proof}

\begin{prop}\label{prop.semisimple}
Let $G=\PSL_2(q)$ and let $\cC$ be the $G$-conjugacy class of a
semisimple element $x$. Then $\cC^2$ contains all the semisimple
elements in $G$. Moreover $1 \in \cC^2$.
\end{prop}
\begin{proof}
Let $z \in G$ be some semisimple element. Let $X,Z \in \SL_2(q)$ be
the pre-images of $x,z$ respectively. Denote
$\al=\tr(X),\ga=\tr(Z)$. By Theorem~\ref{thm.Mac}, there exist
matrices $X',Y',Z' \in \SL_2(q)$ such that $\tr(X')=\tr(Y')=\al$ and
$\tr(Z')=\ga$. By Table~\ref{table.elm.PSL}, $\bar X',\bar Y' \in
\cC$ and so $\bar Z' \in \cC^2$. Hence by Lemma~\ref{lem.conj.tr},
$z \in \cC^2$ as needed.

Moreover, since $\tr(X)=\tr(X^{-1})$ then according to
Table~\ref{table.elm.PSL}, $x^{-1} \in \cC$ as well, and so
$xx^{-1}=1 \in \cC^2$.
\end{proof}

Note that the previous proposition is a specific case of a more
general result of Gow~\cite{Gow}.

\begin{prop}\label{prop.even.unip.ss}
Let $G=\PSL_2(q)$ when $q$ is even and let $\cC$ be the
$G$-conjugacy class of a unipotent element $x$. Then $\cC^2=G$.
\end{prop}
\begin{proof}
Let $z \in G$ be any semisimple element and denote $\ga=\tr(z)$. By
Theorem~\ref{thm.Mac}, there exist matrices $x',y',z' \in G$ such
that $\tr(x')=\tr(y')=0$ and $\tr(z')=\ga$. Thus $x',y' \in \cC$ and
so $z' \in \cC^2$. Hence by Lemma~\ref{lem.conj.tr}, $z \in \cC^2$
as needed.

Moreover, since $x$ has order $2$ then $x^2=1 \in \cC^2$. Without
loss of generality we may assume that $x=\begin{pmatrix} 0 & 1 \\ 1
&0 \end{pmatrix}$ and take $y=\begin{pmatrix} a+1 & a \\ a & a+1
\end{pmatrix} \in \SL_2(q)$ (where $a \ne 0,1$). Then $xy=yx=
\begin{pmatrix} a & a+1 \\ a+1 & a
\end{pmatrix}$. Since $\tr(x)=\tr(y)=\tr(xy)=0$ then $x,y,xy$ are
unipotents. Hence by Lemma~\ref{lem.conj.tr}, $\cC^2$ contains all
the unipotent elements.
\end{proof}

\subsection{Unipotent elements contained in $\cC^2$}
\begin{prop}\label{prop.split.unip}
Let $G=\PSL_2(q)$ and let $\cC$ be the $G$-conjugacy class of a
semisimple split element $x$. Then $\cC^2$ contains all the
unipotent elements in $G$.
\end{prop}
\begin{proof}
By Table~\ref{table.elm.PSL}, we may assume $x$ is an image of a
matrix $X = \begin{pmatrix} a & 1 \\ 0 & a^{-1}\end{pmatrix} \in
\SL_2(q)$ (where $a \ne 0, \pm 1$). Let $Y=\begin{pmatrix} a^{-1} & a^{-1}(1-a) \\
0 & a\end{pmatrix}$, then its image in $G=\PSL_2(q)$ is
$G$-conjugate to $x$, and $XY=\begin{pmatrix} 1 & 1 \\ 0 &
1\end{pmatrix}$. Hence $\cC^2$ contains a unipotent element, and the
result now follows from Lemma~\ref{lem.conj.tr}.
\end{proof}

\begin{prop}\label{prop.non.split.unip}
Let $G=\PSL_2(q)$ when $q$ is odd and let $\cC$ be the $G$-conjugacy
class of a semisimple non-split element $x$ of order greater than
$2$. Then $\cC^2$ contains all the unipotent elements in $G$.
\end{prop}
\begin{proof}
By Table~\ref{table.elm.PSL}, we may assume that $x$ is an image of
a matrix $X = \begin{pmatrix} a & -1 \\ 1 & 0\end{pmatrix} \in
\SL_2(q)$ (where $a \ne 0$). Let $Y=\begin{pmatrix} 0 & 1 \\
-1 & -a\end{pmatrix}$, then its image in $G=\PSL_2(q)$ is
$G$-conjugate to $x$, and $XY=\begin{pmatrix} 1 & 2a \\ 0 &
1\end{pmatrix}$. Hence $\cC^2$ contains a unipotent element, and the
result now follows from Lemma~\ref{lem.conj.tr}.
\end{proof}

\begin{prop}\label{prop.non.split.2}
Let $G=\PSL_2(q)$ when $q$ is odd and let $\cC$ be the $G$-conjugacy
class of a semisimple non-split element $x$ of order $2$. Then
$\cC^2$ does not contain any unipotent element of $G$.
\end{prop}
\begin{proof}
Assume that $\cC^2$ contains a unipotent element $z \in G$. Then
$z=xy$ for some non-split elements $x,y$ of order $2$. Let $X,Y,Z
\in \SL_2(q)$ be the pre-images of $x,y,z$ respectively. Then
$\tr(X)=\tr(Y)=0$ and $\tr(Z)=\pm 2$. Observe that $(0,0,\pm 2)$ is
a \emph{singular} triple. Therefore, by
Theorem~\ref{thm.Mac.singular}, the subgroup $H=\langle x,y \rangle$
is a \emph{structural} subgroup of $G$. By observing the possible
subgroups of $G$ detailed in Table~\ref{table.subgroups}, $H$ cannot
be a subgroup of the Borel subgroup since it contains a non-split
element, and hence cannot contain the unipotent element $z$,
yielding a contradiction.
\end{proof}

\begin{cor}\label{cor.conj.inv}
Let $G=\PSL_2(q)$ when $q$ is odd and let $\cC$ be the $G$-conjugacy
class of a semisimple element $x$ of order $2$. Then
\begin{itemize}
\item If $q \equiv 1 \bmod 4$ then $\cC^2$ contains all the unipotent elements in $G$.
\item If $q \equiv 3 \bmod 4$ then $\cC^2$ does not contain any unipotent element of
$G$.
\end{itemize}
\end{cor}
\begin{proof}
If $q \equiv 1 \bmod 4$ then an element of order $2$ is split and
the result follows from Proposition~\ref{prop.split.unip}. If $q
\equiv 3 \bmod 4$ then an element of order $2$ is non-split and the
result follows from Proposition~\ref{prop.non.split.2}.
\end{proof}

\begin{prop}\label{prop.non.split.even}
Let $G=\PSL_2(q)$ when $q$ is even and let $\cC$ be the
$G$-conjugacy class of a semisimple non-split element $x$. Then
$\cC^2$ does not contain any unipotent element of $G$.
\end{prop}
\begin{proof}
Assume that $\cC^2$ contains a unipotent element $z \in G$. Then
$z=xy$ for some non-split matrices $x,y$. Then $\tr(x)=\tr(y)=\al
\ne 0$ and $\tr(z)=0$. Observe that $(\al,\al,0)$ is a
\emph{singular} triple. Therefore, by
Theorem~\ref{thm.Mac.singular}, the subgroup $H=\langle x,y \rangle$
is a \emph{structural} subgroup of $G$. By observing the possible
subgroups of $G$ detailed in Table~\ref{table.subgroups}, $H$ cannot
be a subgroup of the Borel subgroup since it contains a non-split
element, and hence cannot contain the unipotent element $z$,
yielding a contradiction.
\end{proof}

\subsection{Unipotent conjugacy classes $\cC$ when $q$ is odd}
\begin{prop}\label{prop.G0.unip.triple}
Let $G_0=\SL_2(q)$ when $q=p^e$ is odd, and let $A,B \in G_0$, $A,B
\ne \pm I$ satisfy $\tr(A),\tr(B)\in \{\pm 2\}$. Denote $C=AB$ and
$\ga=\tr(C)$.
\begin{enumerate}
\item If $\tr(A)=\tr(B)$, then $A$ is $G_0$-conjugate to $B$ if and
only if $2-\ga$ is a square in $\mathbb{F}_q$.
\item If $\tr(A)=-\tr(B)$, then $A$ is $G_0$-conjugate to $-B$ if and
only if $2+\ga$ is a square in $\mathbb{F}_q$.
\end{enumerate}
\end{prop}
\begin{proof}
Without loss of generality we may assume that $A=U_1$.

\begin{enumerate}
\item If $B=MU_1M^{-1}$ for some matrix $M=\begin{pmatrix} a & b \\ c
& d \end{pmatrix} \in G_0$, then
\[
    \ga = \tr(C) = \tr(AB) = \tr(U_1MU_1M^{-1}) =
    \tr \begin{pmatrix}  1-ac-c^2 & 1+ac+a^2 \\ -c^2 & 1+ac \end{pmatrix} = 2-c^2,
\]
and so $2-\ga$ is a square in $\mathbb{F}_q$.

\item[] If $B=MU'_1M^{-1}$ for some matrix $M \in G_0$, then
\[
    \ga = \tr(C) = \tr(AB) = \tr(U_1MU'_1M^{-1}) = 2-x^2c^2,
\]
and since $x \in \mathbb{F}_{q^2} \setminus \mathbb{F}_q$ then
$2-\ga$ is a non-square in $\mathbb{F}_q$.

\item Similarly, if $B=M(-U_{1})M^{-1}$ for some matrix $M \in G_0$,
then
\[
    \ga = \tr(C) = \tr(AB) = \tr(U_1M(-U_{1})M^{-1}) = -2+c^2,
\]
and so $2+\ga$ is a square in $\mathbb{F}_q$.

\item[] If $B=M(-U'_{1})M^{-1}$ for some matrix $M \in G_0$, then
\[
    \ga = \tr(C) = \tr(AB) = \tr(U_1M(-U'_{1})M^{-1}) = -2+x^2c^2,
\]
and so $2+\ga$ is a non-square in $\mathbb{F}_q$.
\end{enumerate}
\end{proof}

Therefore, in order to decide whether $C=AB$, where $A,B\in G$ are
$G$-conjugate unipotent elements and $C$ is of order $t\ne p$, one
needs to determine whether for $\ga \in \Tr_q(t)$, $2-\ga$ or
$2+\ga$ is a square in $\mathbb{F}_q$. In
Proposition~\ref{prop.good.G.order} we showed that this is
equivalent to decide whether $t$ is a \emph{$q$-good} order or not
(see Definition~\ref{defn.G.good.order}).

\begin{cor}\label{corr.G.unip.triple}
Let $G=\PSL_2(q)$ when $q=p^e$ is odd. Let $A,B \in G$ be two
elements of order $p$ and assume that the order of $C=AB$ is $t \ne
p$. Then $A$ is $G$-conjugate to $B$ if and only if $t$ is a
\emph{$q$-good} order.
\end{cor}
\begin{proof}
Let $A,B\in G_0=\SL_2(q)$ and $C=AB$ and assume that their images
$\bar{A},\bar{B},\bar{C} \in G=\PSL_2(q)$ have respective orders
$(p,p,t)$, $t \ne p$. Denote $\ga=\tr(C)$. Then $\bar{A}$ and
$\bar{B}$ are unipotent if and only if $A,B \ne \pm I$ and
$\tr(A),\tr(B)\in \{\pm 2\}$. Moreover, $\bar{A}$ and $\bar{B}$ are
$G$-conjugate if and only if either $\tr(A)=\tr(B)$ and $A$ and $B$
are $G_0$-conjugate or $\tr(A)=-\tr(B)$ and $A$ and $-B$ are
$G_0$-conjugate. From Proposition~\ref{prop.G0.unip.triple} we
deduce that $\bar{A}$ and $\bar{B}$ are $G$-conjugate if and only if
$2-\ga$ or $2+\ga$ is a square in $\mathbb{F}_q$. By
Proposition~\ref{prop.good.G.order}, the latter is equivalent to $t$
being a \emph{$q$-good} order.
\end{proof}

\begin{cor}\label{cor.odd.unip.ss}
Let $G=\PSL_2(q)$ when $q$ is odd and let $\cC$ be the $G$-conjugacy
class of a unipotent element $x$. Let $z$ be a semisimple element in
$G$. Then $\cC^2$ contains $z$ if and only if the order of $z$ is
\emph{$q$-good}.
\end{cor}
\begin{proof}
Follows from Corollary~\ref{corr.G.unip.triple} and
Lemma~\ref{lem.conj.tr}.
\end{proof}

\begin{prop}\label{prop.good.unip.triple}
Assume that $q=p^e$ where $p$ is odd and $5 \leq q \neq 9$ and let
$G=\PSL_2(q)$. Then there exist two $G$-conjugate unipotent elements
$A$ and $B$ such that $\langle A, B \rangle=G$.

Moreover, if $C=AB$ is semisimple then the order of $C$ is
\emph{$q$-minimal} and \emph{$q$-good}.
\end{prop}
\begin{proof}
Let $(A,B,C)\in G^3$ be a triple of respective orders $(p,p,t)$ such
that $t \neq p$. If $t$ is not $q$-good then by
Corollary~\ref{corr.G.unip.triple}, $A$ and $B$ are not
$G$-conjugate. If $t$ is not $q$-minimal then the subgroup $\langle
A,B \rangle$ is contained in a subfield subgroup isomorphic to
$\PSL_2(q_1)$ for some proper subfield $\BF_{q_1} \subset \BF_q$
(see~\cite{Mac} and \cite{LR1}).

Now, let $t$ be a $q$-minimal order which is also a $q$-good order.
By Theorem~\ref{thm.Mac}, there exist a triple $(A,B,C)\in G^3$ of
respective orders $(p,p,t)$ such that $ABC=1$. By
Corollary~\ref{corr.G.unip.triple}, $A$ is $G$-conjugate to $B$.

Let $\ga \in \Tr_{q}(t)$, as $t \ne p$ then $\ga \ne \pm 2$ and so
$(\ga \pm 2)^2\ne 0$ implying that $(2,2,\ga)$ is
\emph{non-singular}, hence by Theorem~\ref{thm.Mac.singular},
$\langle A, B \rangle$ is not a structural subgroup of $G$. Since
$t$ is $q$-minimal, then $\langle A, B \rangle$ is not isomorphic
neither to $\PSL_2(q_1)$ nor to $\PGL_2(q_1)$, for some proper
subfield $\BF_{q_1}$ of $\BF_q$. Moreover, if $5< q \ne 9$ then
either $p>5$; or $p=5$ and $e>1$ implying that $t \ne 2,3,5$; or
$p=3$ and $e>2$ implying that $t>5$ (see Table~\ref{table.ord.PSL}).
Therefore, $\langle A,B \rangle$ cannot be a small subgroup, hence
$\langle A,B \rangle = G$.
\end{proof}

\begin{rem}\label{rem.9.335}
Let $G=\PSL_2(9) \cong A_6$, let $A,B$ be two unipotent elements (of
order $3$) and $C=(AB)^{-1}$. Denote the order of $C$ by $t$, then
$t \in \{2,3,4,5\}$.
\begin{itemize}
\item Clearly if $t=2$ or $t=3$ then $\langle A,B \rangle \neq G$
(see Table~\ref{table.ord.PSL}).

\item $t=4$ is not $9$-good and so $A$ and $B$ are not
$G$-conjugate, by Corollary~\ref{corr.G.unip.triple}.

\item $t=5$ is $9$-good, however if $A$ is $G$-conjugate to $B$, then
one can verify that $\langle A,B \rangle \cong A_5$ is a small
subgroup of $G$ (see e.g.~\cite[\S2, Theorem 8.4]{Go}).
\end{itemize}
\end{rem}

\subsection*{Acknowledgement}
The author would like to thank Gili Schul for useful discussions.



\begin{thebibliography}{MM}

\bibitem{CM}
H.S.M. Coxeter, W.O.J. Moser, \textit{Generators and relations for
discrete groups}, Ergebnisse der Mathematik und Ihrer Grenzgebrete,
New Series, no. 14. Springer, Berlin-Gottingen-Heidelberg, 1957.

\bibitem{Di}
L.E. Dickson, \textit{Linear Groups with an Exposition of the Galois
Field Theory}, Teubner (1901).

\bibitem{Do}
L. Dornhoff, {\it Group Representation Theory}, Part A, Marcel
Dekker, 1971.

\bibitem{EG}
E.W. Ellers, N. Gordeev, \textit{On the conjectures of J. Thompson
and O. Ore}, Trans. Amer. Math. Soc. {\bf 350} (1998), 3657�-3671

\bibitem{Ga}
S. Garion, \textit{On Beauville Structures for $\PSL(2,q)$},
arXiv:1003.2792.

\bibitem{Go}
D. Gorenstein, \emph{Finite groups}, Chelsea Publishing Co., New
York, 1980.

\bibitem{Gow}
R. Gow, \textit{Commutators in finite simple groups of Lie type},
Bull. London Math. Soc. {\bf 32} (2000), no. 3, 311--315.


\bibitem{GM}
R. Guralnick, G. Malle, \textit{Products of conjugacy classes and
fixed point spaces}, J. Amer. Math. Soc. {\bf 25} (2012), 77--121.

\bibitem{LR1}
U. Langer, G. Rosenberger, \textit{Erzeugende endlicher projektiver
linearer Gruppen}, Results Math. {\bf 15} (1989), no. 1-2, 119--148.

\bibitem{LR2}
F. Levin, G. Rosenberger, \textit{Generators of finite projective
linear groups. II.}, Results Math. {\bf 17} (1990), no. 1-2,
120--127.

\bibitem{Mac}
A.M. Macbeath, \textit{Generators of the linear fractional groups},
Number Theory (Proc. Sympos. Pure Math., Vol. XII, Houston, Tex.,
1967), Amer. Math. Soc., Providence, R.I. (1969), 14--32.

\bibitem{Mar}
C. Marion, \textit{Triangle groups and $\PSL_2(q)$}, J. Group Theory
\textbf{12} (2009), 689--708.

\bibitem{Sch}
G. Schul, \textit{Expansion in finite simple groups}.

\bibitem{Su}
M. Suzuki, \textit{Group Theory I}, Springer, Berlin (1982).


\end{thebibliography}
\end{document}